\documentclass[reqno,12pt]{amsart}
\usepackage{amscd,amsfonts,amssymb}
\usepackage[T2A]{fontenc}
\usepackage[cp1251]{inputenc}
\numberwithin{equation}{section}
\newtheorem{Theorem}{Theorem}[section]
\newtheorem{Lemma}{Lemma}[section]
\newtheorem{Corollary}{Corollary}[section]

\theoremstyle{definition}

\theoremstyle{remark}

\newtheorem{Example}{Example}[section]

\newcommand{\mes}{\mathop{\rm mes}\nolimits}

\author{A.A. Kon'kov}
\address{Department of Differential Equations,
Faculty of Mechanics and Mathematics,
Mo\-s\-cow Lo\-mo\-no\-sov State University,
Vorobyovy Gory,
Moscow, 119992 Russia.
Center of Nonlinear Problems of Mathematical Physics,
RUDN University,
Miklukho-Maklaya str. 6,
Moscow, 119991 Russia.
}
\email{konkov@mech.math.msu.su}
\author{A.E. Shishkov}
\address{
Center of Nonlinear Problems of Mathematical Physics,
RUDN University,
Mi\-klu\-k\-ho-Mak\-la\-ya str. 6,
Moscow, 117198 Russia.
}
\email{aeshkv@yahoo.com}
\title[On global solutions]{On global solutions of quasilinear second-order elliptic inequalities}
\date{}

\begin{document}

\begin{abstract}
We consider the inequality
$$
	- \operatorname{div} A (x, \nabla u)
	\ge
	f (u)
	\quad
	\mbox{in } {\mathbb R}^n,
$$
where $n \ge 2$ and $A$ is a Caratheodory function such that
$$
	C_1
	|\xi|^p
	\le
	\xi
	A (x, \xi)
	\quad
	\mbox{and}
	\quad
	|A (x, \xi)|
	\le
	C_2
	|\xi|^{p-1}
$$
with some constants
$C_1 > 0$,
$C_2 > 0$,
and
$p > 1$
for almost all
$x \in {\mathbb R}^n$
and for all
$\xi \in {\mathbb R}^n$.
Our aim is to find exact conditions on the function $f$ guaranteeing that any non-negative solution of this inequality is identically zero.
\end{abstract}

\maketitle

\section{Introduction}

We study non-negative solutions of the inequality
\begin{equation}
	- \operatorname{div} A (x, \nabla u)
	\ge
	f (u)
	\quad
	\mbox{in } {\mathbb R}^n,
	\label{1.1g}
\end{equation}
where $n \ge 2$ and $A$ is a Caratheodory function such that
$$
	C_1
	|\xi|^p
	\le
	\xi
	A (x, \xi)
	\quad
	\mbox{and}
	\quad
	|A (x, \xi)|
	\le
	C_2
	|\xi|^{p-1}
$$
with some constants $C_1 > 0$, $C_2 > 0$, and $p > 1$ for almost all $x \in {\mathbb R}^n$ and for all $\xi \in {\mathbb R}^n$.
It is also assumed that the function $f : [0, \infty) \to [0, \infty)$ does not decrease on the interval $[0, \varepsilon]$ for some real number $\varepsilon > 0$. 

We say that $u \in W_{p, loc}^1 ({\mathbb R}^n)$ is a solution of~\eqref{1.1g} if 
$f (u) \in L_{1, loc} ({\mathbb R}^n)$
and
$$
	\int_{
		{\mathbb R}^n
	}
	A (x, \nabla u)
	\nabla \varphi
	\,
	dx
	\ge
	\int_{
		{\mathbb R}^n
	}
	f (u)
	\varphi
	\,
	dx
$$
for any non-negative function $\varphi \in C_0^\infty ({\mathbb R}^n)$.

A partial case of~\eqref{1.1g} is the inequality
\begin{equation}
	- \Delta_p u \ge f (u)
	\quad
	\mbox{in } {\mathbb R}^n,
	\label{1.1p}
\end{equation}
where
$
	\Delta_p u = \operatorname{div} (|\nabla u|^{p-2} \nabla u)
$
is the $p$-Laplace operator.

Without loss of generality, we can assume that solutions of~\eqref{1.1g} satisfy the relation
\begin{equation}
	\operatorname*{ess\,inf}\limits_{
		{\mathbb R}^n
	}
	u
	=
	0;
	\label{1.2}
\end{equation}
otherwise we replace $u$ by $u - \alpha$, where
\begin{equation}
	\alpha
	=
	\operatorname*{ess\,inf}\limits_{
		{\mathbb R}^n
	}
	u.
	\label{1.3}
\end{equation}
In so doing, the left-hand side of~\eqref{1.1g} obviously does not change and the right-hand side transforms to $f (u + \alpha)$.

The absence of nontrivial global solutions of differential equations and inequalities or, in other words, the blow-up phenomenon, traditionally attracts the attention of mathematicians~[1--11]. 
We obtain exact conditions on the function $f$ guaranteeing that any non-negative solution of~\eqref{1.1g}, \eqref{1.2} is identically zero. 
Let us note that the only relevant case is $n > p$. Really, in the case of $n \le p$, any non-negative solution of the inequality
\begin{equation}
	- \operatorname{div} A (x, \nabla u)
	\ge
	0
	\quad
	\mbox{in } {\mathbb R}^n
	\label{1.4}
\end{equation}
is a constant~\cite{MPbook}.
For $n > p$, in papers~\cite{ASComPDE, AShal}, it was shown that~\eqref{1.1p} has no positive solutions if 
$$
	\operatorname*{lim\,inf}\limits_{\zeta \to + 0}
	\frac{
		f (\zeta)
	}{
		\zeta^{n (p - 1) / (n - p)}
	}
	>
	0.
$$
We manage to strengthen this 
statement
(see Theorems~\ref{T2.1} and~\ref{T2.2}).

\section{Main results}

\begin{Theorem}\label{T2.1}
Let $n > p$ and
\begin{equation}
	\int_0^\varepsilon
	\frac{
		f (\zeta)
		\,
		d\zeta
	}{
		\zeta^{1 + n (p - 1) / (n - p)}
	}
	=
	\infty.
	\label{T2.1.1}
\end{equation}
Then any non-negative solution of~\eqref{1.1g}, \eqref{1.2} is identically zero.
\end{Theorem}

\begin{Theorem}\label{T2.2}
Let $n > p$ and
\begin{equation}
	\int_0^\varepsilon
	\frac{
		f (\zeta)
		\,
		d\zeta
	}{
		\zeta^{1 + n (p - 1) / (n - p)}
	}
	<
	\infty.
	\label{T2.2.1}
\end{equation}
Then problem~\eqref{1.1p}, \eqref{1.2} has a positive solution.
\end{Theorem}

Theorems~\ref{T2.1} and~\ref{T2.2} are proved in Section~\ref{proof}. 
Now we demonstrate their applications.

\begin{Example}\label{E2.1}
Consider the inequality
\begin{equation}
	- \operatorname{div} A (x, \nabla u)
	\ge
	u^\lambda
	\quad
	\mbox{in } {\mathbb R}^n,
	\label{E2.1.1}
\end{equation}
where $n > p$ and $\lambda$ is a real number.
Replacing if necessary the function $u$ with $u - \alpha$, where $\alpha$ is defined by~\eqref{1.3} and applying Theorem~\ref{T2.1}, we conclude that, for 
\begin{equation}
	\lambda
	\le
	\frac{
		n (p - 1)
	}{
		n - p
	},
	\label{E2.1.2}
\end{equation}
any non-negative solution of~\eqref{E2.1.1} is identically zero.
Condition~\eqref{E2.1.2} is exact. Really, by Theorem~\ref{T2.2}, in the case of
$$
	\lambda
	>
	\frac{
		n (p - 1)
	}{
		n - p
	},
$$
the inequality
$$
	- \Delta_p u
	\ge
	u^\lambda
	\quad
	\mbox{in } {\mathbb R}^n,
$$
has a positive solution.

We note that~\eqref{E2.1.2} coincides with the analogous condition obtained in~\cite{ASComPDE, AShal}.
Earlier in~\cite{MPbook} it was shown that the blow-up of solutions of~\eqref{E2.1.1} is occur if
$$
	p - 1
	<
	\lambda
	\le
	\frac{
		n (p - 1)
	}{
		n - p
	}.
$$
\end{Example}

\begin{Example}\label{E2.2}
Let us examine the critical exponent 
$
	\lambda 
	= 
	n (p - 1) / (n - p)
$
in~\eqref{E2.1.2}. Consider the inequality
\begin{equation}
	- \operatorname{div} A (x, \nabla u)
	\ge
	u^{
		n (p - 1) / (n - p)
	}
	\log^\mu
	\left(
		e
		+
		\frac{1}{u}
	\right)
	\quad
	\mbox{in } {\mathbb R}^n,
	\label{E2.2.1}
\end{equation}
where $n > p$. For $u = 0$, we assume that the right-hand side of~\eqref{E2.2.1} equals to zero.

According to Theorem~\ref{T2.1}, if
\begin{equation}
	\mu \ge - 1,
	\label{E2.2.2}
\end{equation}
then any non-negative solution of~\eqref{E2.2.1} is identically zero.
At the same time, Theorem~\ref{T2.2} implies that, in the case of
$$
	\mu < -1,
$$
the inequality 
$$
	- \Delta_p u
	\ge
	u^{
		n (p - 1) / (n - p)
	}
	\log^\mu
	\left(
		e
		+
		\frac{1}{u}
	\right)
	\quad
	\mbox{in } {\mathbb R}^n,
$$
has a positive solution.
Thus, condition~\eqref{E2.2.2} is exact.
\end{Example}

\section{Proof of Theorems~\ref{T2.1} and~\ref{T2.2}}\label{proof}

In this section, by $C$ and $\sigma$ we mean various positive constants that can depend only on $p$, $n$, $\varepsilon$, $\lambda$, and the ellipticity constants $C_1$ and $C_2$. Also let $B_r$ and $S_r$ be the open ball and the sphere of radius $r > 0$ centered at zero and $A_r = B_{2 r} \setminus B_r$.
As is customary, by $\chi_\Omega$ we mean the characteristic function of a set $\Omega \subset {\mathbb R}^n$, i.e.
$$
	\chi_\Omega (x)
	=
	\left\{
		\begin{array}{ll}
			1,
			&
			x
			\in
			\Omega,
			\\
			0,
			&
			x
			\not\in
			\Omega.
		\end{array}
	\right.
$$

We need a few preliminary statements.

\begin{Lemma}[Generalized Kato's inequality]\label{L3.1}
Let $v \in W_{p, loc}^1 (\Omega)$ be a solution of the inequality
$$
	\operatorname{div} A (x, \nabla v)
	\ge
	a (x)
	\quad
	\mbox{in } \Omega,
$$
where $\Omega \subset {\mathbb R}^n$ is a non-empty open set and $a \in L_{1, loc} (\Omega)$.
Then the function $v_+ = \chi_{\Omega_+} v$ is a solution of the inequality
$$
	\operatorname{div} A (x, \nabla v_+)
	\ge
	\chi_{\Omega_+} (x)
	a (x)
	\quad
	\mbox{in } \Omega,
$$
where $\Omega_+ = \{ x \in \Omega : v (x) > 0 \}$. 
\end{Lemma}

Lemma~\ref{L3.1} is proved in~\cite[Lemma~4.2]{meJMAA_2007}.
Applying Lemma~\ref{L3.1} with $v = \varepsilon - u$, we arrive at Corollary~\ref{C3.1} given below.

\begin{Corollary}\label{C3.1}
Let $u$ be a solution of~\eqref{1.1g}. Then 
$
	u_\varepsilon 
	= 
	\chi_{\Omega_\varepsilon} 
	u 
	+ 
	(1 - \chi_{\Omega_\varepsilon})
	\varepsilon
$
is a solution of the inequality
$$
	- \operatorname{div} A (x, \nabla u_\varepsilon)
	\ge
	\chi_{\Omega_\varepsilon} (x)
	f (u)
	\quad
	\mbox{in } {\mathbb R}^n,
$$
where $\Omega_\varepsilon = \{ x \in {\mathbb R}^n : u (x) < \varepsilon \}$.
\end{Corollary}

\begin{Lemma}[Weak Harnack inequality]\label{L3.2}
Let $n > p$ and $u \ge 0$ be a solution of~\eqref{1.4}.
Then 
$$
	\left(
		\frac{
			1
		}{
			\mes B_{2 r}
		}
		\int_{
			B_{2 r}
		}
		u^\lambda
		\,
		dx
	\right)^{1 / \lambda}
	\le
	C
	\operatorname*{ess\,inf}\limits_{
		B_r
	}
	u
$$
for all $\lambda \in (0, n (p - 1) / (n - p))$ and $r \in (0, \infty)$.
\end{Lemma}

\begin{Lemma}\label{L3.3}
Let $u \ge 0$ be a solution of the inequality
$$
	- \operatorname{div} A (x, \nabla u)
	\ge
	a (x)
	\quad
	\mbox{in } {\mathbb R}^n,
$$
where $a \in L_{1, loc} ({\mathbb R}^n)$ is a non-negative function.
If $u^\lambda \in L_{1, loc} ({\mathbb R}^n)$ for some $\lambda \in (p - 1, \infty)$, then
$$
	\frac{1}{\mes B_r}
	\int_{B_r}
	a (x)
	\,
	dx
	\le
	C
	r^{- p}
	\left(
		\frac{
			1
		}{
			\mes A_r
		}
		\int_{
			A_r
		}
		u^\lambda
		\,
		dx
	\right)^{(p - 1) / \lambda}
$$
for all $r \in (0, \infty)$.
\end{Lemma}

\begin{Lemma}\label{L3.4}
Let $n > p$ and $u \ge 0$ be a solution of~\eqref{1.4}, \eqref{1.2}. Then
$$
	\lim_{r \to \infty}
	\frac{
		\mes \Omega_\varepsilon \cap B_r
	}{
		\mes B_r
	}
	=
	1,
$$
where $\Omega_\varepsilon = \{ x \in {\mathbb R}^n : u (x) < \varepsilon \}$.
\end{Lemma}

Proof of Lemmas~\ref{L3.2} and~\ref{L3.3} can be found in~\cite{Serrin, Trudinger, TW}  and~\cite{AM}, respectively.
Lem\-ma~\ref{L3.4} is given in~\cite[Lemma 3.1]{AMAdvMath}. Proof of this lemma follows immediately from Lemma~\ref{L3.2}.
Indeed, by Lemma~\ref{L3.2}, we have
$$
	\varepsilon
	\left(
		\frac{
			\mes B_{2 r} \setminus \Omega_\varepsilon
		}{
			\mes B_{2 r}
		}
	\right)^{1 / \lambda}
	\le
	\left(
		\frac{
			1
		}{
			\mes B_{2 r}
		}
		\int_{
			B_{2 r}
		}
		u^\lambda
		\,
		dx
	\right)^{1 / \lambda}
	\le
	C
	\operatorname*{ess\,inf}\limits_{
		B_r
	}
	u
$$
with some $\lambda \in (0, n (p - 1) / (n - p))$ for all $r \in (0, \infty)$. 
In the limit as $r \to \infty$, this yields
$$
	\lim_{r \to \infty}
	\frac{
		\mes B_{2 r} \setminus \Omega_\varepsilon
	}{
		r^n
	}
	=
	0.
$$

Corollary~\ref{C3.1} and Lemmas~\ref{L3.2} and~\ref{L3.3} with $\lambda \in (p - 1, p) \cap (0, n (p - 1) / (n - p))$ lead to the following assertion.

\begin{Corollary}\label{C3.2}
Let $n > p$ and $u \ge 0$ be a solution of~\eqref{1.1g}.
Then 
$$
	\frac{1}{\mes B_r}
	\int_{\Omega_\varepsilon \cap B_r}
	f (u)
	\,
	dx
	\le
	C
	r^{- p}
	\left(
		\operatorname*{ess\,inf}\limits_{
			B_r
		}
		u_\varepsilon
	\right)^{p - 1}
$$
for all $r \in (0, \infty)$, 
where 
$\Omega_\varepsilon = \{ x \in {\mathbb R}^n : u (x) < \varepsilon \}$ 
and 
$
	u_\varepsilon 
	= 
	\chi_{\Omega_\varepsilon} 
	u 
	+ 
	(1 - \chi_{\Omega_\varepsilon})
	\varepsilon.
$
\end{Corollary}

\begin{proof}[Proof of Theorem~$\ref{T2.1}$]
Arguing by contradiction, we assume that $u \ge 0$ is a non-zero solution of~\eqref{1.1g}, \eqref{1.2}.
In view of Lemma~\ref{L3.2}, this solution is positive almost everywhere in ${\mathbb R}^n$.
By Lemma~\ref{L3.4}, there exists a real number $r_0 > 0$ such that
$$
	\mes \Omega_\varepsilon \cap B_r 
	\ge
	C
	r^n
$$
and
$$
	\mes \Omega_\varepsilon \cap A_r 
	\ge
	C
	r^n
$$
for all $r \ge r_0$, where
$\Omega_\varepsilon = \{ x \in {\mathbb R}^n : u (x) < \varepsilon \}$.

We denote $r_i = 2^i r_0$, $i = 1,2,\ldots$.
Also let
$$
	E (r)
	=
	\int_{\Omega_\varepsilon \cap B_r}
	f (u)
	\,
	dx,
	\quad
	r > 0.
$$
Corollary~\ref{C3.2} implies that
$$
	\frac{
		1
	}{
		\mes B_{r_i}
	}
	\int_{
		\Omega_\varepsilon \cap B_{r_i}
	}
	f (u)
	\,
	dx
	\le
	C
	r^{- p}
	\left(
		\operatorname*{ess\,inf}\limits_{
			B_{r_i}
		}
		u_\varepsilon
	\right)^{p - 1},
	\quad
	i = 1,2,\ldots,
$$
whence we immediately obtain
\begin{equation}
	\operatorname*{ess\,inf}\limits_{
		\Omega_\varepsilon \cap B_{r_i}
	}
	u
	=
	\operatorname*{ess\,inf}\limits_{
		B_{r_i}
	}
	u_\varepsilon
	\ge
	\sigma
	\left(
		\frac{
			E (r_i)
		}{
			r_i^{n - p} 
		}
	\right)^{1 / (p - 1)},
	\quad
	i = 1,2,\ldots.
	\label{PT2.1}
\end{equation}
By~\eqref{1.2}, this allows us to assert that
\begin{equation}
	\lim_{i \to \infty}
	\frac{
		E (r_i)
	}{
		r_i^{n - p} 
	}
	=
	0.
	\label{PT2.3}
\end{equation}
Since $f$ is a non-decreasing function on the interval $[0, \varepsilon]$,
formula~\eqref{PT2.1} yields
$$
	f
	\left(
		\operatorname*{ess\,inf}\limits_{
			\Omega_\varepsilon \cap B_{r_i}
		}
		u
	\right)
	\ge
	f
	\left(
		\sigma
		r_i^{- (n - p) / (p - 1)}
		E^{1 / (p - 1)} (r_i)
	\right),
	\quad
	i = 1,2,\ldots.
$$
In view of the evident inequality
$$
	\frac{
		1
	}{
		\operatorname{mes} \Omega_\varepsilon \cap A_{r_{i-1}}
	}
	\int_{
		\Omega_\varepsilon \cap A_{r_{i-1}}
	}
	f (u)
	\,
	dx
	\ge
	f
	\left(
		\operatorname*{ess\,inf}\limits_{
			\Omega_\varepsilon \cap B_{r_i}
		}
		u
	\right)
$$
and the choice of the real number $r_0$, this implies that
$$
	\int_{
		\Omega_\varepsilon \cap A_{r_{i-1}}
	}
	f (u)
	\,
	dx
	\ge
	C
	r_i^n
	f
	\left(
		\sigma
		r_i^{- (n - p) / (p - 1)}
		E^{1 / (p - 1)} (r_i)
	\right)
$$
or, in other words,
$$
	E (r_i) - E (r_{i-1})
	\ge
	C
	r_i^n
	f
	\left(
		\sigma
		r_i^{- (n - p) / (p - 1)}
		E^{1 / (p - 1)} (r_i)
	\right),
	\quad
	i = 1,2,\ldots.
$$
The last estimate, in turn, yields
\begin{equation}
	\frac{
		E (r_i) - E (r_{i-1})
	}{
		E^{n / (n - p)} (r_i)
	}
	\ge
	C
	h
	\left(
		\sigma
		r_i^{- (n - p) / (p - 1)}
		E^{1 / (p - 1)} (r_i)
	\right),
	\quad
	i = 1,2,\ldots,
	\label{PT2.2}
\end{equation}
where
$$
	h (\zeta)
	=
	\frac{
		f (\zeta)
	}{
		\zeta^{n (p - 1) / (n - p)}
	},
	\quad
	\zeta > 0.
$$
There obviously are sequences of integers $0 < s_j < l_j \le s_{j + 1}$, $j = 1,2,\ldots$, such that
$$
	\frac{
		E (r_{i - 1})
	}{
		r_{i - 1}^{n - p} 
	}
	>
	\frac{
		E (r_i)
	}{
		r_i^{n - p} 
	}
$$
if $i \in \cup_{j=1}^\infty (s_j, l_j]$ and
$$
	\frac{
		E (r_{i - 1})
	}{
		r_{i - 1}^{n - p} 
	}
	\le
	\frac{
		E (r_i)
	}{
		r_i^{n - p} 
	}
$$
otherwise.
Multiplying~\eqref{PT2.2} by the inequality
$$
	1
	\ge
	\frac{
		r_{i - 1}^{- (n - p) / (p - 1)}
		E^{1 / (p - 1)} (r_{i - 1})
		-
		r_i^{- (n - p) / (p - 1)}
		E^{1 / (p - 1)} (r_i)
	}{
		r_{i - 1}^{- (n - p) / (p - 1)}
		E^{1 / (p - 1)} (r_{i - 1})
	},
$$
one can conclude that
\begin{align}
	\frac{
		E (r_i) - E (r_{i-1})
	}{
		E^{n / (n - p)} (r_i)
	}
	\ge
	{}
	&
	\frac{
		C
		h
		\left(
			\sigma
			r_i^{- (n - p) / (p - 1)}
			E^{1 / (p - 1)} (r_i)
		\right)
	}{
		r_{i - 1}^{- (n - p) / (p - 1)}
		E^{1 / (p - 1)} (r_{i - 1})
	}
	\nonumber
	\\
	&
	{}
	\times
	\left(
		r_{i - 1}^{- (n - p) / (p - 1)}
		E^{1 / (p - 1)} (r_{i - 1})
		-
		r_i^{- (n - p) / (p - 1)}
		E^{1 / (p - 1)} (r_i)
	\right)
	\label{PT2.4}
\end{align}
for all $i \in \cup_{j=1}^\infty (s_j, l_j]$.
Since $E$ is a non-decreasing function, we have
$$
	\frac{
		2^{n - p}
		E (r_i)
	}{
		r_i^{n - p} 
	}
	\ge
	\frac{
		E (r_{i - 1})
	}{
		r_{i - 1}^{n - p} 
	}
	>
	\frac{
		E (r_i)
	}{
		r_i^{n - p} 
	}
$$
for all $i \in \cup_{j=1}^\infty (s_j, l_j]$. Hence,
$$
	\int_{
		E (r_{i - 1})
	}^{
		E (r_i)
	}
	\frac{
		d\zeta
	}{
		\zeta^{n / (n - p)}
	}
	\ge
	\frac{
		E (r_i) - E (r_{i-1})
	}{
		E^{n / (n - p)} (r_i)
	}
$$
and
\begin{align*}
	&
	\frac{
		h
		\left(
			\sigma
			r_i^{- (n - p) / (p - 1)}
			E^{1 / (p - 1)} (r_i)
		\right)
	}{
		r_{i - 1}^{- (n - p) / (p - 1)}
		E^{1 / (p - 1)} (r_{i - 1})
	}
	\left(
		r_{i - 1}^{- (n - p) / (p - 1)}
		E^{1 / (p - 1)} (r_{i - 1})
		-
		r_i^{- (n - p) / (p - 1)}
		E^{1 / (p - 1)} (r_i)
	\right)
	\\
	&
	\qquad
	{}
	\ge
	2^{- (n - p) / (p - 1)} 
	\int_{
		r_i^{- (n - p) / (p - 1)}
		E^{1 / (p - 1)} (r_i)
	}^{
		r_{i - 1}^{- (n - p) / (p - 1)}
		E^{1 / (p - 1)} (r_{i - 1})
	}
	\frac{
		\tilde h (\sigma \zeta)
	}{
		\zeta
	}
	\,
	d\zeta
\end{align*}
for all $i \in \cup_{j=1}^\infty (s_j, l_j]$, where
$$
	\tilde h (\zeta)
	=
	\inf_{
		(
			2^{- (n - p) / (p - 1)} 
			\zeta, 
			\,
			\zeta
		)
	}
	h,
	\quad
	\zeta > 0.
$$
Combining this with~\eqref{PT2.4}, we obtain
$$
	\int_{
		E (r_{i - 1})
	}^{
		E (r_i)
	}
	\frac{
		d\zeta
	}{
		\zeta^{n / (n - p)}
	}
	\ge
	C
	\int_{
		r_i^{- (n - p) / (p - 1)}
		E^{1 / (p - 1)} (r_i)
	}^{
		r_{i - 1}^{- (n - p) / (p - 1)}
		E^{1 / (p - 1)} (r_{i - 1})
	}
	\frac{
		\tilde h (\sigma \zeta)
	}{
		\zeta
	}
	\,
	d\zeta
$$
for all $i \in \cup_{j=1}^\infty (s_j, l_j]$, whence it follows that
\begin{equation}
	\sum_{j=1}^\infty
	\int_{
		E (r_{s_j})
	}^{
		E (r_{l_j})
	}
	\frac{
		d\zeta
	}{
		\zeta^{n / (n - p)}
	}
	\ge
	C
	\sum_{j=1}^\infty
	\int_{
		r_{l_j}^{- (n - p) / (p - 1)}
		E^{1 / (p - 1)} (r_{l_j})
	}^{
		r_{s_j}^{- (n - p) / (p - 1)}
		E^{1 / (p - 1)} (r_{s_j})
	}
	\frac{
		\tilde h (\sigma \zeta)
	}{
		\zeta
	}
	\,
	d\zeta.
	\label{PT2.5}
\end{equation}
Taking into account the inequalities
$$
	E (r_{s_{j+1}})
	\ge
	E (r_{l_j}),
	\quad
	j = 1,2,\ldots,
$$
we have
$$
	\int_{
		E (r_{s_1})
	}^\infty
	\frac{
		d\zeta
	}{
		\zeta^{n / (n - p)}
	}
	\ge
	\sum_{j=1}^\infty
	\int_{
		E (r_{s_j})
	}^{
		E (r_{l_j})
	}
	\frac{
		d\zeta
	}{
		\zeta^{n / (n - p)}
	}.
$$
At the same time,~\eqref{PT2.3} and the inequalities
$$
	r_{s_{j+1}}^{- (n - p) / (p - 1)}
	E^{1 / (p - 1)} (r_{s_{j+1}})
	\ge
	r_{l_j}^{- (n - p) / (p - 1)}
	E^{1 / (p - 1)} (r_{l_j}),
	\quad
	j = 1,2,\ldots,
$$
allows us to assert that
$$
	\sum_{j=1}^\infty
	\int_{
		r_{l_j}^{- (n - p) / (p - 1)}
		E^{1 / (p - 1)} (r_{l_j})
	}^{
		r_{s_j}^{- (n - p) / (p - 1)}
		E^{1 / (p - 1)} (r_{s_j})
	}
	\frac{
		\tilde h (\sigma \zeta)
	}{
		\zeta
	}
	\,
	d\zeta
	\ge
	\int_0^{
		r_{s_1}^{- (n - p) / (p - 1)}
		E^{1 / (p - 1)} (r_{s_1})
	}
	\frac{
		\tilde h (\sigma \zeta)
	}{
		\zeta
	}
	\,
	d\zeta.
$$
Hence,~\eqref{PT2.5} implies the estimate
\begin{equation}
	\int_{
		E (r_{s_1})
	}^\infty
	\frac{
		d\zeta
	}{
		\zeta^{n / (n - p)}
	}
	\ge
	C
	\int_0^{
		r_{s_1}^{- (n - p) / (p - 1)}
		E^{1 / (p - 1)} (r_{s_1})
	}
	\frac{
		\tilde h (\sigma \zeta)
	}{
		\zeta
	}
	\,
	d\zeta.
	\label{PT2.6}
\end{equation}
It is easy to see that
$$
	\int_{
		E (r_{s_1})
	}^\infty
	\frac{
		d\zeta
	}{
		\zeta^{n / (n - p)}
	}
	=
	\frac{
		n - p
	}{
		p
	}
	E^{ - p / (n - p)} (r_{s_1})
	<
	\infty.
$$
In so doing,
$$
	\tilde h (\sigma \zeta)
	\ge
	\frac{
		C
		f (\sigma \zeta)
	}{
		\zeta^{n (p - 1) / (n - p)}
	}
$$
for all $\zeta > 0$ in a neighborhood of zero, because $f$ is a non-decreasing function on the interval $[0, \varepsilon]$. Thus, in accordance with~\eqref{T2.1.1} we obtain
$$
	\int_0^{
		r_{s_1}^{- (n - p) / (p - 1)}
		E^{1 / (p - 1)} (r_{s_1})
	}
	\frac{
		\tilde h (\sigma \zeta)
	}{
		\zeta
	}
	\,
	d\zeta
	=
	\infty.
$$
This contradicts~\eqref{PT2.6}. 
\end{proof}

\begin{proof}[Proof of Theorem~$\ref{T2.2}$]
Without loss of generality, it can be assumed that $f (\zeta) > 0$ for all $\zeta \in (0, \varepsilon)$; otherwise we replace the function $f (\zeta)$ by
$
	\max
	\{
		f (\zeta),
		\zeta^{1 + n (p - 1) / (n - p)}
	\}.
$

Let us find the required solution of~\eqref{1.1p}, \eqref{1.2} in the form
$$
	u (x) = w_\delta (|x|),
$$
where
$$
	w_\delta (r)
	=
	\int_r^\infty
	\left(
		\frac{
			1
		}{
			\zeta^{n - 1}
		}
		\int_0^\zeta
		\xi^{n - 1}
		f
		\left(
			\varepsilon
			\left(
				1
				+
				\frac{\xi}{\delta}
			\right)^{- (n - p) / (p - 1)}
		\right)
		d\xi
	\right)^{1 / (p - 1)}
	d\zeta.
$$
Here, $\delta > 0$ is a real number which will be defined later.
We note that the right-hand side of the last equality is well defined for all $\delta \in (0, \infty)$ and, moreover,
\begin{equation}
	\lim_{\delta \to +0}
	\sup_{
		[0, \infty)
	}
	w_\delta
	=
	0
	\label{PT2.2.1}
\end{equation}
and
\begin{equation}
	w_\delta (r)
	\le
	C
	\delta^{n / (p - 1)}
	r^{- (n - p) / (p - 1)}
	\left(
		\int_0^\varepsilon
		\frac{
			f (\zeta)
			\,
			d\zeta
		}{
			\zeta^{1 + n (p - 1) / (n - p)}
		}
	\right)^{1 / (p - 1)}
	\label{PT2.2.2}
\end{equation}
for all $r \in (0, \infty)$.
Really, by change of variables $\zeta = \varepsilon (1 + \xi / \delta)^{- (n - p) / (p - 1)}$, we have
\begin{align*}
	&
	\int_0^\infty
	\xi^{n - 1}
	f
	\left(
		\varepsilon
		\left(
			1
			+
			\frac{\xi}{\delta}
		\right)^{- (n - p) / (p - 1)}
	\right)
	d\xi
	\\
	&
	\quad
	{}
	=
	\frac{p - 1}{n - p}
	\frac{\delta^n}{\varepsilon}
	\int_0^\varepsilon
	\left(
		\left(
			\frac{\varepsilon}{\zeta}
		\right)^{(p - 1) / (n - p)}
		-
		1
	\right)^{n - 1}
	\left(
		\frac{\varepsilon}{\zeta}
	\right)^{(p - 1) / (n - p) + 1}
	f (\zeta)
	\,
	d\zeta
	\\
	&
	\qquad
	{}
	\le
	\frac{p - 1}{n - p}
	\delta^n
	\varepsilon^{n (p - 1) / (n - p)}
	\int_0^\varepsilon
	\frac{
		f (\zeta)
		\,
		d\zeta
	}{
		\zeta^{1 + n (p - 1) / (n - p)}
	}
	<
	\infty,
\end{align*}
whence it follows that
\begin{align*}
	w_\delta (r)
	&
	\le
	\int_r^\infty
	\left(
		\frac{
			1
		}{
			\zeta^{n - 1}
		}
		\int_0^\infty
		\xi^{n - 1}
		f
		\left(
			\varepsilon
			\left(
				1
				+
				\frac{\xi}{\delta}
			\right)^{- (n - p) / (p - 1)}
		\right)
		d\xi
	\right)^{1 / (p - 1)}
	d\zeta
	\\
	&
	{}
	\le
	C
	\delta^{n / (p - 1)}
	\left(
		\int_0^\varepsilon
		\frac{
			f (\zeta)
			\,
			d\zeta
		}{
			\zeta^{1 + n (p - 1) / (n - p)}
		}
	\right)^{1 / (p - 1)}
	\int_r^\infty
	\frac{
		d\zeta
	}{
		\zeta^{(n - 1) / (p - 1)}
	}
\end{align*}
for all $r \in (0, \infty)$.
This proves~\eqref{PT2.2.2}.

To establish the validity of~\eqref{PT2.2.1}, we note that
\begin{equation}
	\lim_{\delta \to +0}
	\sup_{
		[1, \infty)
	}
	w_\delta
	=
	0
	\label{PT2.2.3}
\end{equation}
in accordance with estimate~\eqref{PT2.2.2}.
At the same time, taking into account the equality
$$
	w_\delta (r)
	=
	w_\delta (1)
	+
	\int_r^1
	\left(
		\frac{
			1
		}{
			\zeta^{n - 1}
		}
		\int_0^\zeta
		\xi^{n - 1}
		f
		\left(
			\varepsilon
			\left(
				1
				+
				\frac{\xi}{\delta}
			\right)^{- (n - p) / (p - 1)}
		\right)
		d\xi
	\right)^{1 / (p - 1)}
	d\zeta,
$$
we obtain 
\begin{align}
	w_\delta (r)
	&
	{}
	\le
	w_\delta (1)
	+
	\int_0^1
	\left(
		\frac{
			1
		}{
			\zeta^{n - 1}
		}
		\int_0^\zeta
		\xi^{n - 1}
		f
		\left(
			\varepsilon
			\left(
				1
				+
				\frac{\xi}{\delta}
			\right)^{- (n - p) / (p - 1)}
		\right)
		d\xi
	\right)^{1 / (p - 1)}
	d\zeta
	\nonumber
	\\
	&
	{}
	\le
	w_\delta (1)
	+
	\left(
		\int_0^1
		f
		\left(
			\varepsilon
			\left(
				1
				+
				\frac{\xi}{\delta}
			\right)^{- (n - p) / (p - 1)}
		\right)
		d\xi
	\right)^{1 / (p - 1)}
	\label{PT2.2.4}
\end{align}
for all $r \in [0, 1]$.
Since $f : [0, \infty) \to [0, \infty)$ is a non-decreasing function on the interval $[0, \varepsilon]$ satisfying condition~\eqref{T2.2.1}, we have $f (\zeta) \to 0$ as $\zeta \to +0$. Therefore, the second summand on the right in~\eqref{PT2.2.4} tends to zero as $\delta \to +0$ by Lebesgue's dominated convergence theorem, while $w_\delta (1) \to 0$ as $\delta \to +0$ in accordance with~\eqref{PT2.2.2}.
Hence,~\eqref{PT2.2.4} implies that
$$
	\lim_{\delta \to +0}
	\sup_{
		[0, 1]
	}
	w_\delta
	=
	0.
$$
Combining this with~\eqref{PT2.2.3}, we readily arrive at~\eqref{PT2.2.1}.

Let us fix some $\delta > 0$ such that
\begin{equation}
	\varepsilon
	\left(
		1
		+
		\frac{r}{\delta}
	\right)^{- (n - p) / (p - 1)}
	\ge
	w_\delta (r)
	\label{PT2.2.5}
\end{equation}
for all $r \in [0, \infty)$. Such a real number $\delta > 0$ obviously exists. 
Indeed, in view of~\eqref{PT2.2.1}, there exists $\delta_1 > 0$ such that
$$
	\varepsilon
	\left(
		1
		+
		\frac{r}{\delta}
	\right)^{- (n - p) / (p - 1)}
	\ge
	\varepsilon
	2^{- (n - p) / (p - 1)}
	\ge
	w_\delta (r)
$$
for all $0 \le r < \delta \le \delta_1$.
At the same time, in view of~\eqref{PT2.2.2}, there exists $\delta_2 > 0$ such that
$$
	\varepsilon
	\left(
		1
		+
		\frac{r}{\delta}
	\right)^{- (n - p) / (p - 1)}
	\ge
	\varepsilon
	2^{- (n - p) / (p - 1)}
	\delta^{(n - p) / (p - 1)}
	r^{- (n - p) / (p - 1)}
	\ge
	w_\delta (r)
$$
for all $0 < \delta \le \delta_2$ and $r \ge \delta$.
Hence, we can take $\delta = \min \{ \delta_1, \delta_2 \}$.

By direct differentiation, it can be verified that
$$
	- \frac{
		1
	}{
		r^{n - 1}
	}
	\frac{d}{d r}
	\left(
		r^{n - 1}
		\left|
			\frac{d w_\delta}{d r}
		\right|^{p - 2}
		\frac{d w_\delta}{d r}
	\right)
	=
	f
	\left(
		\varepsilon
		\left(
			1
			+
			\frac{r}{\delta}
		\right)^{- (n - p) / (p - 1)}
	\right)
$$
for all $r \in (0, \infty)$, whence in accordance with~\eqref{PT2.2.5} we have
$$
	- \frac{
		1
	}{
		r^{n - 1}
	}
	\frac{d}{d r}
	\left(
		r^{n - 1}
		\left|
			\frac{d w_\delta}{d r}
		\right|^{p - 2}
		\frac{d w_\delta}{d r}
	\right)
	\ge
	f
	\left(
		w_\delta
	\right).
$$
for all $r \in (0, \infty)$. 
Therefore, the function $u$ is a classical solution of the inequality
$$
	- \Delta_p u
	\ge
	f (u)
	\quad
	\mbox{in } {\mathbb R}^n \setminus \{ 0 \}.
$$
Multiplying this by a non-negative function $\varphi \in C_0^\infty ({\mathbb R}^n)$ and integrating over ${\mathbb R}^n \setminus B_r$, where $r > 0$ is a real number, we obtain
\begin{align}
	\int_{
		{\mathbb R}^n \setminus B_r
	}
	f (u)
	\varphi
	\,
	dx
	\le
	-
	\int_{
		{\mathbb R}^n \setminus B_r
	}
	\Delta_p u
	\varphi
	\,
	dx
	=
	{}
	&
	-
	\int_{
		S_r
	}
	|\nabla u|^{p-2} \frac{\partial u}{\partial \nu}
	\varphi
	\,
	d S_r
	\nonumber
	\\
	&
	{}
	+
	\int_{
		{\mathbb R}^n \setminus B_r
	}
	|\nabla u|^{p-2} 
	\nabla u \nabla \varphi
	\,
	dx,
	\label{PT2.2.6}
\end{align}
where $\nu$ is the unit normal vector at a point of the sphere $S_r$ outer to ${\mathbb R}^n \setminus B_r$.
It is easy to see that
\begin{align*}
	|\nabla u (x)|
	=
	&
	\left|
		\frac{d w_\delta (r)}{d r}
	\right|_{r = |x|}
	=
	\left(
		\frac{
			1
		}{
			|x|^{n - 1}
		}
		\int_0^{|x|}
		\xi^{n - 1}
		f
		\left(
			\varepsilon
			\left(
				1
				+
				\frac{\xi}{\delta}
			\right)^{- (n - p) / (p - 1)}
		\right)
		d\xi
	\right)^{1 / (p - 1)}
	\\
	&
	{}
	\le
	\left(
		\int_0^{|x|}
		f
		\left(
			\varepsilon
			\left(
				1
				+
				\frac{\xi}{\delta}
			\right)^{- (n - p) / (p - 1)}
		\right)
		d\xi
	\right)^{1 / (p - 1)}
	\to
	0
	\quad
	\mbox{as } x \to 0,
\end{align*}
whence we have
$$
	\left|
		\int_{
			S_r
		}
		|\nabla u|^{p-2} \frac{\partial u}{\partial \nu}
		\varphi
		\,
		d S_r
	\right|
	\le
	\int_{
		S_r
	}
	|\nabla u|^{p - 1} 
	|\varphi|
	\,
	d S_r
	\to
	0
	\quad
	\mbox{as } r \to +0
$$
and
$$
	\int_{
		{\mathbb R}^n \setminus B_r
	}
	|\nabla u|^{p-2} 
	\nabla u \nabla \varphi
	\,
	dx
	\to
	\int_{
		{\mathbb R}^n
	}
	|\nabla u|^{p-2} 
	\nabla u \nabla \varphi
	\,
	dx
	\quad
	\mbox{as } r \to +0.
$$
Thus, to complete the proof, it remains to pass in~\eqref{PT2.2.6} to the limit as $r \to +0$.
\end{proof}

\end{document}